\documentclass[11pt,letterpaper]{amsart}
\usepackage{url,amsmath,amsthm, amssymb}
\usepackage[margin=1in]{geometry}
\usepackage[pdfencoding=auto, psdextra, colorlinks=true, linkcolor = blue, urlcolor=blue, citecolor=blue,pagebackref,breaklinks]{hyperref}
\usepackage{bookmark}
\usepackage[all]{mathtools}
\usepackage{makecell}
\usepackage{mathrsfs}
\usepackage{mathdots}
\usepackage{amsxtra,amscd}
\usepackage{amsrefs}
\usepackage{hyperref}
\usepackage{enumerate, enumitem}
\usepackage{arydshln}
\usepackage{multirow}
\usepackage{bbm} 
\usepackage[all,arc]{xy}

 \usepackage{amssymb,amsrefs}
\usepackage{amsmath}
\usepackage{amscd,latexsym}
\usepackage{bookmark}
 \usepackage{amssymb,amsfonts,bm}
\usepackage[all,arc]{xy}
\usepackage{enumerate}
\usepackage{mathrsfs}
\usepackage{amscd}
\usepackage{tikz}
\usepackage{tikz-cd}
\usepackage{amsmath}
\usepackage{latexsym}
\usepackage{mathdots}
\usepackage{amsrefs}
\usepackage{graphicx}
\usepackage{mathtools}
\usepackage{leftidx}
\usepackage{tensor}
\usepackage{dsfont}
\usepackage{tikz-cd} 
\usepackage[new]{old-arrows}
\usepackage{xcolor}
\usepackage{bbm}
\usepackage{enumitem}
\usepackage{cleveref}
\usepackage{xcolor}

\allowdisplaybreaks

\theoremstyle{plain}
\newtheorem{theorem}{Theorem}[section]
\newtheorem{taggedtheoremx}{Theorem}
\newenvironment{taggedtheorem}[1]
 {\renewcommand\thetaggedtheoremx{#1}\taggedtheoremx}
 {\endtaggedtheoremx}
\newtheorem{corollary}[theorem]{Corollary}
\newtheorem{induction hypothesis}[theorem]{Induction Hypothesis}

\newtheorem{lemma}[theorem]{Lemma}
\newtheorem{proposition}[theorem]{Proposition}
\numberwithin{equation}{section}

\theoremstyle{definition}
 
  \newtheorem*{data availability}{Data availability}
  \newtheorem*{conflicts of interest}{Conflicts of interest}

\theoremstyle{remark}
\newtheorem{remark}[theorem]{Remark}

\newtheorem*{acknowledgements}{Acknowledgments}

\makeatletter
\newcommand{\mytag}[2]{%
  \text{#1}%
  \@bsphack
  \begingroup
    \@onelevel@sanitize\@currentlabelname
    \edef\@currentlabelname{%
      \expandafter\strip@period\@currentlabelname\relax.\relax\@@@%
    }%
    \protected@write\@auxout{}{%
      \string\newlabel{#2}{%
        {#1}%
        {\thepage}%
        {\@currentlabelname}%
        {\@currentHref}{}%
      }%
    }%
  \endgroup
  \@esphack
}

\begin{document}

\title[Local coefficients and  Gelfand-Graev representations]{Local coefficients and Gelfand-Graev representations for non-split covers on SL$(2)$}

\author{Yeongseong Jo}
\address{Department of Mathematics Education, Ewha Womans University, Seoul 03760, Republic of Korea}
\email{\href{mailto:jo.59@buckeyemail.osu.edu}{jo.59@buckeyemail.osu.edu};\href{mailto:yeongseong.jo@ewha.ac.kr}{yeongseong.jo@ewha.ac.kr}}

\subjclass[2020]{Primary 11F70; Secondary 22E50}
\keywords{Bushnell–Kutzko's types and covers, Gelfand-Graev representations, Langlands-Shahidi local coefficients}

\begin{abstract}
A purely local approach has been developed by Krishnamurthy and Kutzko to compute Langlands-Shahidi local coefficient for ${\rm SL}(2)$ via 
types and covers \`{a} la Bushnell-Kutzko. In this paper, we extend their method to the non-split case and complete their project.
We also study the algebraic structure of Gelfand-Graev representations, which generalizes
the results of Chan-Savin and Mishra-Pattanayak to ${\rm SL}(2)$ over non-archimidean local fields without any restriction on the characteristic. 
\end{abstract}

\maketitle

\section{Introduction}

Throughout this paper $F$ will denote a non-arhimedean local field and we let $G={\rm SL}_2(F)$.
The purpose of this paper is to illustrate an algebraic method of achieving an explicit formula for the {\it Langlands--Shahidi  local coefficients}
and computing the isotypic subspace of {\it Gelfand--Graev representations} of $G$ by means of {\it types} and {\it non-split covers} \`{a} la  Bushnell--Kutzko. 
At a first glance, the results seem to be rather tangential to each other, but two topics are commonly concerned with what is known as {\it Whittaker models},
which play an important role in the automorphic forms and representation theory of linear reductive groups due to the nature of its relation to $L$-functions,
notably in the {\it Langlands--Shahidi method} and {\it Rankin--Selberg integrals} as well. Specifically, local coefficient by definition is a constant of proportionality arising from the uniqueness of Whittaker models \cite{Rod72}, whereas the Gelfand--Graev representation is given by the dual of the Whittaker space. In the meantime, the theory of types and covers provides a systematic way of analyzing the structure of smooth representations of linear reductive $p$-adic groups via the representation theory of compact open subgroups (cf. \Cref{TC}).

\par
This paper consists of two parts. Our first main result of this paper is a complete expression of local coefficients via types and covers, restricting to the case of $F^{\times}$
as a Levi subgroup of $G$. Specifically, the character of $F^{\times}$ is quadratic and ramified (``non-split"). As a matter of fact, this case is more challenging than others. 
We resolve the only remaining case of non-split cover, which is addressed in the work by Krishnamurthy and Kutzko \cite[\S 3.4 Remark 2]{KK}.
The reader should consult \S \ref{TC} and \S \ref{SL2-LC} for details about the unexplained notations in the statement below.

\begin{taggedtheorem}{A}[\Cref{LocalCoefficient}]
\label{mainA}
Let $\psi$ be a fixed additive character of level $0$. Let $\widetilde{\eta}$ be a non-trivial ramified quadratic character of $F^{\times}$ of level $n_{\eta}$ with $\eta=\widetilde{\eta} \restriction_{\mathfrak{o}^{\times}_F}$. Then we have
\[
  C_{\psi}(s,\widetilde{\eta},{\rm w}_0)=\widetilde{\eta}(-\varpi_F^{n_{\eta}})\tau(\eta,\psi,\varpi_F^{-n_{\eta}})q^{-n_{\eta}(s-1)}.
\]
\end{taggedtheorem}

The local coefficient has a great influence on the development of contemporary number theory, and is closed related to the theory of local factors {\it \`{a} la} {\it the Langlands-Shahidi method} \cite{Sha78,Sha81,Sha84}. Shahidi subsequently defined the so-called {\it Langlands--Shahidi $\gamma$-factors} inductively so that the local coefficient factorizes as a product of such $\gamma$-factors.
Another briefly aforementioned method for constructing local exterior square factors via the theory of integral representations is developed by  Bump and Friedberg in late 1980's \cite{BF89}.
By definition, it is a proportionality factor between an integral and its dual integral related to each other through the theory of either Fourier transforms or intertwining operators \cite{Mat15}. On the perspective of  the local Langlands correspondence, the resulting local factors ought to be the same, probably up to some normalizations of certain Haar measures.
In practice,  this flavor of comparisons is surely non-trivial to answer, and is known for only a handful amount of cases; for example, we refer the reader to \cite{Kap15,Sha84} for Rankin--Selberg local factors, and \cite{AKMSS21,BP21} for Asai local factors. In this regard, Bump and Friedberg \cite[Conjecture 4]{BF89} predicted that the equality should be valid in the context of local exterior square factors attached to irreducible admissible representation of ${\rm GL}_{n}(F)$. Our \Cref{mainA} confirms that the equality between local factors \cite[Conjecture 4]{BF89}  obtains from the Langlands--Shahidi method and those obtained via
Bump--Friedberg integrals holds unconditionally at least for ${\rm GL}_1(F) \cong F^{\times}$.

\par
Our strategy to tackle the problem of computing local coefficient on $F^{\times}$ inside $G$ is going back to at least the pioneering work of Casselman \cite{Cas80},
where local coefficients in the context of unramified principal series representations are explicitly computed. The computation there relies on finding the effect of
the intertwining operator on the subspace of vectors fixed by the Iwahori subgroup. The role played by the trivial representation of the Iwahori subgroup in Casselman's trick can be regarded 
as an extreme incident of the theory of types and covers. Previously, types and covers are adapted by Krishnamurthy and Kutzko \cite{KK} for {\it split covers} of  $G$, or in other words, the case where the character is not an unramified twist of a quadratic character.
Afterwords, the author and Krishnamurthy \cite{JK21} dealt with local coefficients associated to covers of a homogeneous pair of irreducible supercuspidal representations $(\pi_1, \pi_2)$ of the Levi subgroup ${\rm GL}_n(F) \times {\rm GL}_n(F) $ embedded in ${\rm GL}_{2n}(F)$.

\par
Our next aim of this paper is to express the principal series block of 
the Gelfand--Graev representation as a cyclic module over the Iwahori--Hecke algebra. 
As before, the reader is advised to refer to Section \ref{TC} for Hecke algebras, and undefined terminology and to Section \ref{GGrep} for Gelfand--Graev representations
in the following statement.

\begin{taggedtheorem}{B}[\Cref{Gelfand-Grev}]
\label{mainB}
Let $\psi$ be a fixed additive character of level $0$. Let $\widetilde{\eta}$ be a non-trivial ramified quadratic character of $F^{\times}$ of level $n_{\eta}$ with $\eta=\widetilde{\eta} \restriction_{\mathfrak{o}^{\times}_F}$. As $\mathcal{H}(G,\lambda_{\eta})-$modules, we obtain isomorphisms
\[
({\text {\rm c-ind}}_U^G{\psi})^{\lambda_{\eta}} \cong \mathcal{H}(G,\lambda_{\eta} ) \otimes_{{\mathcal{H}(K,\lambda_{\eta})}}  \mathbb{C}_{-1}.
\]
\end{taggedtheorem}

The structure of Gelfand--Graev representation was originally treated by Chan and Savin for unramified principal series blocks of split reductive groups \cite{CS18}, and was further refined by Mishra and Pattanaya for principal series blocks for connected reductive groups over $F$ whose residue characteristic is large enough \cite{MP21}. Shortly after, Gao, Gurevich, and Karasiewicz extended Chan and Savin's result for linear groups \cite{CS18} to the Iwahori fixed vectors in the Gelfand-Graev representation of covering groups as a module over the Iwahori-Hecke algebra  \cite{GGK22}.

\par
A part of reasons why Mishra and Pattanaya \cite{MP21} primarily considered limited characteristics of $F$ is that they made a reduction to depth-zero cases and then to finite group cases.
In doing so, they were able to determine the generator of isotypical space for the cover of Gelfand-Graev representations over finite fields, which in turn admits the $1$-dimensional sign representation.
On the other hand, our proof takes elements from Chan and Savin's  \cite{CS18} unconditional and direct argument. Similarly to Chan and Savin,  we laboriously compute the outcome
of Hecke operators on test functions from which we find the generator possessing the $1$-dimensional sign representation (cf. \Cref{Whittaker-Hecke}).

\par
In principle, all the computations boil down to understanding the behavior of the intertwining or Hecke operators on certain test functions (cf. \Cref{intertwining}, \Cref{Whittaker-Hecke}).
We expect that those test functions that naturally occur in the theory of types and cover open a new chapter for other situation, and consider the present paper a groundwork for that direction. Benefiting from \cite{JK21}, it is also our belief that the structure of the isotypic component of the space of compactly supported Whittaker functions as modules over Iwahori-Hecke algebra, is responsible for computing exterior square local coefficients in more general settings, remarkably, Siegel Levi subgroups isomorphic to ${\rm GL}_n(F)$ lying inside ${\rm Sp}_{2n}(F)$. It will be therefore very interesting to see if our robust arguement can be carried out to parabolically induced representations of symplectic groups in $2n$-variables. The author plans to investigate them in a future study.

\section{Non-split Covers and Intertwining Operators}
\label{Intertwining O}

\subsection{Types and Covers}
\label{TC}

Let $F$ be a non-archimedean local field with its residual finite field $\mathbb{F}_q$ and denote by $p$ the characteristic of $\mathbb{F}_q$.
The base field $F$ is a finite extension of $\mathbb{Q}_p$ or $\mathbb{F}_p((t))$, called a {\it $p$-adic field} in characteristic 0,
or a {\it local function field} in characteristic $p > 0$. 
Let $\mathfrak{o}_F$ be its ring of integers, $\mathfrak{p}_F$ its maximal ideal.
We fix a generator $\varpi_F$ of $\mathfrak{p}_F$ and normalize the absolute value $|\cdot|$ of $F$ so that
$|\varpi_F|=q_F^{-1}=q^{-1}$. When there is no possibility of confusion, we sometimes drop the subscripts, while working over a fixed $F$.
We let $G={\rm SL}_2(F)$. Let $B$ be the subgroup of $F$-points of the Borel subgroup of upper triangular matrices. Then $B=TU$, where
\[
 T= \left\{  \begin{pmatrix} a & 0 \\ 0 & a^{-1} \end{pmatrix} \, \middle| \, a \in F^{\times} \right\};
\quad
  U= \left\{ u(x):= \begin{pmatrix} 1 & x \\ 0 & 1 \end{pmatrix} \, \middle| \, x \in F \right\}.
\]
To be specific, we identify $T$ with  $F^{\times}$ when no confusion can arise. Let $\overline{B}=T\overline{U}$ denote the opposite Borel subgroup of lower triangular matrices, where
\[
 \overline{U}= \left\{ \overline{u}(x):= \begin{pmatrix} 1 & 0 \\ x & 1 \end{pmatrix} \, \middle| \, x \in F \right\}.
\]
Let $K=G(\mathfrak{o}_F)$ be a maximal compact subgroup and let ${\rm \bf W}=\{ {\rm I}_2,  {\rm w}_0\}$, where
\[
 {\rm I}_2=\begin{pmatrix} 1  &  0 \\ 0 & 1 \end{pmatrix}  \quad \text{and} \quad {\rm w}_0=\begin{pmatrix} 0  & -1 \\ 1 & 0 \end{pmatrix}
\]
are two fixed distinct representatives of cosets in $N_G(T) \slash T$. We let $\ell$ denotes the length function on ${\rm \bf W}$ defined by $\ell({\rm I}_2)=0$ and $\ell({\rm w}_2)=1$. The character $\nu_s$ is given by the formula
\[
  \nu_s \left( \begin{pmatrix} a & 0 \\ 0 & a^{-1} \end{pmatrix}  \right)=|a|^s.
\]
We may view it as a character of $B$ by extending $\nu_s$ trivially to $U$. Let $\delta$ denote the modulus 
character of $B$; explicitly,
\[
  \delta \left( \begin{pmatrix} 1 & x \\ 0 & 1 \end{pmatrix}  \begin{pmatrix} a & 0 \\ 0 & a^{-1} \end{pmatrix}  \right)=|a|^2.
\]
For a subgroup $H$ of $G$ and $g \in G$, let $H^g$ denote $g^{-1}Hg$. If $\tau$ is a representation of $H$, let $\tau^g$ be 
a representation of $H^g$ such that $\tau^g(h)=\tau(ghg^{-1})$, $h \in H^g$. Let ${\textbf 1}_H$ denote a trivial character on $H$.

\par
For any topological group $H$, we write $\widehat{H}$ to denote the group of continuous homomorphism from $H$ to $\mathbb{C}^{\times}$.
Particularly, we are interested in the character group $\widehat{\mathfrak{o}}_F^{\times}$.
Let $T(\mathfrak{o}_F):=T \cap K$. Specifically, $T(\mathfrak{o}_F) \cong \mathfrak{o}_F^{\times}$.
Then $T \slash T(\mathfrak{o}_F)$ is a free abelian group of rank $1$. 
Let $X(T)$ denote the group of continuous homomorphisms of $T$ into $\mathbb{C}^{\times}$ which are trivial on $T(\mathfrak{o}_F)$ -- 
called the group of unramified characters of $T$. Moreover $X(T)$ is equipped with the structure of a complex variety 
whose ring of regular functions is $\mathbb{C}[T \slash T(\mathfrak{o}_F)] \cong \mathbb{C}[t,t^{-1}]$.

\par
By a {\it a cuspidal pair} in $G$, we mean a pair $(M,\sigma)$ in $G$,
where $M$ is either $T$ or $G$,
and $\sigma$ is a supercuspidal representation of $M$.
Two such pairs $(M_i,\sigma_i)$, $i=1,2$, are said to be
{\it inertially equivalent} if there exist $g \in G$ and an unramified character $\mu$ of $G$
such that $M_2=M_1^g=g^{-1}M_1g$, and $\sigma_2$ is equivalent to the representation 
$\sigma_1^g \otimes \mu : x \mapsto \sigma_1(gxg^{-1})\mu(x)$ of $M_2$.
We denote by $[(M,\sigma)]_G$ the $G$-inertial equivalence class of a cuspidal pair $(M,\sigma)$ in $G$.
Let $\mathfrak{B}(G)$ denote the set of inertial equivalence classes of cuspidal pairs in $G$. 
In particular, we let $\widetilde{\eta}$ be a character of $F^{\times}$ and we put $\eta=\widetilde{\eta} \restriction_{\mathfrak{o}^{\times}_F}$.
Let $\widetilde{\eta}_i$, $i=1,2$, be characters of $T$. Then $\widetilde{\eta}_2$ is $G$-inertially equivalent to 
$\widetilde{\eta}_1$ if and only if there exists $s \in \mathbb{C}$ such that $\widetilde{\eta}_2=\widetilde{\eta}_1^{\pm 1}\nu_s$. 

\par
It is a fundamental result of Bernstein (cf. \cite[\S 1.4]{Kut04}) that $\mathfrak{R}(G)$ of smooth complex representations of $G$ decomposes into a product of 
full subcategories
\[
  \mathfrak{R}(G) \cong \prod_{\mathfrak{s} \in \mathfrak{B}(G)}  \mathfrak{R}^{\mathfrak{s}}(G).
\]
We divide the equivalence classes into so-called
 \begin{enumerate}[label=$(\mathrm{\roman*})$]
 \item\label{sup} supercuspidal blocks;
  \item\label{pri} principal series blocks.
 \end{enumerate}
\ref{sup} For each irreducible supercuspidal representation $\sigma$ of $G$, we write $\mathfrak{s}(\sigma)$ for the equivalence class of $\sigma$ in $\mathfrak{R}(G)$.
We let $\mathfrak{R}^{\mathfrak{s}(\sigma)}(G)$ be the full subcategory of $ \mathfrak{R}(G)$ whose objects are isomorphic to sums of copies of $\sigma$.
\ref{pri} Let $\chi$ be a character of $B$, and $\iota^{G}_{B}$ the functor of normalized parabolic induction.
We let $\mathcal{F}_B(\chi)$ denote the space of $\iota^{G}_{B}(\chi)$. To be precise, $\mathcal{F}_B(\chi)$ is the space of smooth functions 
$f : G \rightarrow \mathbb{C}$ that satisfy
\[
 f(utg)=\delta^{1/2}(t)\chi(t)f(g)
\]
for any $u \in U$, $t \in T$, and $g \in G$, and the action of $\rho$ on $\mathcal{F}_B(\chi)$ is by right translation, namely $(\rho(h)\cdot f)(g):=f(gh)$.
We denote $G$-intertially equivalent class by $\mathfrak{s}_{\eta}$.
We let $\mathfrak{R}^{\mathfrak{s}_{\eta}}(G)$ be the full subcategory of $\mathfrak{R}(G)$ whose irreducible objects are 
exactly those that occur as a subquotient of some $\mathcal{F}_Q(\widetilde{\eta} \otimes \nu_s)$, where
$Q$ is either $B$ or $\overline{B}$ \cite[\S 3.2, P.228]{KK}.

\par
With $\widetilde{\eta}_i$, $i=1,2$, as above, $\widetilde{\eta}_1$ is $T$-inertially equivalent to $\widetilde{\eta}_2$
if and only if there exists $s \in \mathbb{C}$ such that $\widetilde{\eta}_2=\widetilde{\eta}_1\nu_s$.
Let $\mathfrak{t}_{\eta}$ be the corresponding $T$-inertially equivalent class. 
Let $ \mathfrak{R}^{\mathfrak{t}_{\eta}}(T)$ be the full subcategory of $\mathfrak{R}(T)$
whose object $(\pi,V)$ has the property that $\pi(x)v=\eta(x)v$ for all $x \in T(\mathfrak{o}_F)$ and $v \in V$.
Then the category $\mathfrak{R}(T)$ similarly decomposes as
a product of its subcategories $\mathfrak{R}^{\mathfrak{t}_{\eta}}(T)$:
\[
\mathfrak{R}(T) \cong  \prod_{\eta \in \widehat{\mathfrak{o}}_F^{\times}} \mathfrak{R}^{\mathfrak{t}_{\eta}}(T).
\]

\par
Let $J$ be a compact open subgroup of $G$, let $(\lambda,W)$ be a smooth irreducible representation of $J$, and write $(\check{\lambda},\check{W})$
for the contragredient representation. Then $\mathcal{H}(G,\lambda)$ is the space of compactly supported function $f : G \rightarrow \mathrm{End}_{\mathbb{C}}(\check{W})$ that satisfy
\[
  f(hxk)=\check{\lambda}(h)f(x)\check{\lambda}(k),\quad x \in G,\;\; h,k \in J.
\]
It is a unital (associative) algebra with respect to the standard convolution operation
\[
  f_1 \star f_2(y)=\int_{G} f_1(x) f_2(x^{-1}y) \, dx, \quad f_1,f_2 \in \mathcal{H}(G,\lambda), \;\; y \in G,
\]
where we normalized the Haar measure on $G$ such that ${\rm vol}(J)=1$.

\par
An element $x \in G$ is called {\it intertwine} if there is a non-zero $J \cap x^{-1}Jx$-homomorphism between $\lambda$ and the conjugate representation $\lambda^x$.
It is equivalent to saying that the double coset $JxJ$ supports a non-zero function in $\mathcal{H}(G,\lambda)$ \cite[\S 2.2]{Kut04}.
Specifically, when $\lambda$ is $1$-dimensional, $\lambda(xkx^{-1})=\lambda(k)$ for any $k \in J \cap x^{-1}Jx$.
We denote by $\mathcal{I}_G(\lambda)$ the set of elements in $G$ which intertwine $\lambda$.

\par
A pair $(J,\lambda)$ is said to be a {\it type} for $\mathfrak{s} \in \mathfrak{B}(G)$, or simply a $\mathfrak{s}$-{\it type}  if 
for every irreducible object $(\pi,V) \in \mathfrak{R}(G)$, we have $(\pi,V) \in \mathfrak{R}^{\mathfrak{s}}(G)$
if and only if $\pi$ contains $\lambda$, that is to say, the space of $\lambda$-covariants $V_{\lambda}:=\mathrm{Hom}_J(W,V)$ is non-trivial.
For $a \in \mathrm{End}_{\mathbb{C}}(\check{W})$, $a^{\vee}$ denote the transpose of $a$ with respect to
the canonical pairing between $W$ and $\check{W}$.
We similarly define the Hecke algebra $\mathcal{H}(G,\check{\lambda})$. There is a canonical anti-isomorphism $f \mapsto \check{f}$ from
$\mathcal{H}(G,\lambda) \rightarrow \mathcal{H}(G,\check{\lambda})$ given by $\check{f}(g)=(f(g^{-1}))^{\vee}$.
There is a natural left $\mathcal{H}(G,\lambda)-{\rm Mod}$ structure
(also denoted as $\pi$) given by
\[
  (f\star \phi)(w):=(\pi(f)\phi)(w)=\int_G \pi(g)\phi(f(g)^{\vee}w) \,dg
\]
for $\phi \in V_{\lambda}$, $w \in W$, and $f \in \mathcal{H}(G,\lambda)$.
For  $\mathfrak{s}$-type $(J,\lambda)$, the map $V \mapsto V_{\lambda}$ induces an equivalence of 
categories $\mathfrak{R}^{\mathfrak{s}}(G) \cong \mathcal{H}(G,\lambda)-{\rm Mod}$.

\par
Given $\eta$ and $\widetilde{\eta}$ as above, set $n_{\eta}=1$ if $1+\mathfrak{p}_F \subseteq \ker \eta$.
We let $n_{\eta}$ to be the smallest number $n$ so that $1+\mathfrak{p}_F^n \subset \ker \eta$;
otherwise it is defined to be the smallest positive integer $n$ so that $1+\mathfrak{p}_F^n \subseteq \ker \eta$.
In particular, if $\widetilde{\eta}$ is unramified, $n_{\eta}=1$.
The compact open subgroup is given by
\[
  J_{\eta}=\left\{ \begin{pmatrix} c_{11} & c_{12} \\ c_{21} & c_{22} \end{pmatrix} \in G \, \middle| \, c_{11},c_{22} \in \mathfrak{o}_F^{\times},  c_{12} \in \mathfrak{o}_F,
  c_{21} \in \mathfrak{p}_F^{n_{\eta}}  \right\},
\]
and $\lambda_{\eta}$ is a function on $J_{\eta}$ given by
\[
 \lambda_{\eta} \left(  \begin{pmatrix} c_{11} & c_{12} \\ c_{21} & c_{22} \end{pmatrix} \right)=\eta(c_{11}).
\]
It follows from \cite{Kut04} that the pair $(J_{\eta},\lambda_{\eta})$ is a $G$-cover for $(T(\mathfrak{o}_F),\eta)$.
We recall certain crucial properties of a $G$-cover $(J_{\eta},\lambda_{\eta})$:
 \begin{enumerate}[label=$(\mathrm{\alph*})$]
 \item\label{Iwahori} (Iwahori Factorization) $J_{\eta}=(J_{\eta} \cap \overline{U})T(\mathfrak{o}_F)(J_{\eta} \cap U)$.
  \item\label{kernel} The representation $\lambda_{\eta}$ is trivial on $J_{\eta} \cap \overline{U}$ and $J_{\eta} \cap U$, while $\lambda_{\eta} \restriction_{T(\mathfrak{o}_F)}=\eta$.
  The pair $(J_{\eta},\lambda_{\eta})$ is a type for $\mathfrak{s}_{\eta}$.
   \item\label{injection} There is a support preserving injective algebra map
$     t_Q : \mathcal{H}(T,\eta) \rightarrow \mathcal{H}(G,\lambda_{\eta})$
   that realizes the parabolic induction functor $\iota_Q^G$ at the level of Hecke algebras.
    It means that the following diagram commutes:
    \[
    \xymatrixcolsep{4pc}
\xymatrixrowsep{2pc}
    \xymatrix{
 \mathfrak{R}^{\mathfrak{t}_{\eta}}(T) \ar[d]_{\iota_Q^G} \ar[r]^{\hspace{-2em}\cong}
& {\mathcal H}(T,\eta){\rm -Mod} \ar[d]^{(t_Q)_{\ast}} \\
\mathfrak{R}^{\mathfrak{s}_{\eta}}(G)  \ar[r]^{\hspace{-2em}\cong} & {\mathcal H}(G,\lambda_{\eta}){\rm -Mod}
}
\]
where $(t_Q)_{\ast}$ is a right adjoint of the restriction functor $t_Q^{\ast} : {\mathcal H}(G,\lambda_{\eta}){\rm -Mod} \rightarrow {\mathcal H}(T,\eta){\rm -Mod}$.
 \end{enumerate}

We oftentimes identify ${\mathcal H}(T,\eta)$ as a sub-algebra of $\mathcal{H}(G,\lambda_{\eta})$ using the embedding $t_Q$. The cover $(J_{\eta},\lambda_{\eta})$
is said to be a {\it split cover}, when $\eta^2 \neq 1$, and $(J_{\eta},\lambda_{\eta})$ is called a {\it non-split cover}, otherwise. Among non-split covers, local coefficients and Gelfand--Graev representations in the frame work of unramified characters $\widetilde{\eta}$ is well understood since the seminal work of Casselman \cite{Cas80}.
For this reason, we fix a non-trivial ramified quadratic character $\widetilde{\eta}$ of $F^{\times}$ with $\eta=\widetilde{\eta} \restriction_{\mathfrak{o}^{\times}_F}$ so that $\widetilde{\eta}^2=1$, and only focus on these types of characters in the rest of this paper.

\subsection{Intertwining Operators}
\label{sec-intertwining}

We write $\mathcal{F}(s,\chi)$ to denote the induced space $\mathcal{F}_B  (\chi \otimes \nu_s)$.
We define $\mathbb{C}_{\lambda_{\eta}}$ to be the space of complex numbers for which $J_{\eta}$ acts on $\mathbb{C}_{\lambda_{\eta}}$ by $j \cdot w=\lambda_{\eta}(j)w$.
Throughout the rest of the paper, we put $W=\mathbb{C}_{\lambda_{\eta}}$. Let $V^{\lambda_{\eta}}$ be the $\lambda_{\eta}$-isotypic subspace of $V$.
As $\lambda_{\eta}$ is one-dimensional, it is straightforward from \cite[(2.13)]{BK98} that there is a natural isomorphism $V_{\lambda_{\eta}} \cong V_{\lambda_{\eta}} \otimes_{\mathbb{C}} W \cong V^{\lambda_{\eta}}$ given by $\phi \otimes w \mapsto \phi(w)$. 
The functions $f_{{\rm I}_2,\widetilde{\eta}}$ and $f_{{\rm w}_0,\widetilde{\eta}}$  in $\mathcal{F}(s,\widetilde{\eta})^{\lambda_{\eta}}$ are given by
\[
  f_{{\rm I}_2,\widetilde{\eta}}(g)=
  \begin{cases}
  \widetilde{\eta}\nu_s\delta^{1/2}(b) \lambda_{\eta}(j) & \text{if} \quad g=bj \in BJ_{\eta} \\
  0 & \text{otherwise,}
  \end{cases}
\]
and
\[
  f_{{\rm w}_0,\widetilde{\eta}}(g)=
  \begin{cases}
 \widetilde{\eta}\nu_s\delta^{1/2}(b) \lambda_{\eta}(j) & \text{if} \quad g=b{\rm w}_0j \in B{\rm w}_0J_{\eta} \\
  0 & \text{otherwise.}
  \end{cases}
\]
For ${\rm w} \in {\rm \bf W}$, let $T_{\rm w} \in \mathcal{H}(G,\lambda_{\eta})$ denote the function supported on the double coset $J_{\eta}{\rm w}J_{\eta}$
and given by the formula $T_{\rm w}(h{\rm w}k)=\check{\lambda}(h)\check{\lambda}(k)$, $h,k \in J_{\eta}$.
Let $\mathcal{H}(K,\lambda_{\eta})$ be the sub-algebra in $\mathcal{H}(G,\lambda_{\eta})$ spanned by the function $T_{\rm w}$ for ${\rm w} \in {\rm \bf W}$.
Just as in \cite[\S 3.3]{Kut04}, we let $T^{\ast}_{\rm w}$ be the normalized Hecke operator so that 
\[
T^{\ast}_{\rm w}=\epsilon^{\ell({\rm w})/2}{\rm vol}(J_{\eta} {\rm w} J_{\eta})^{-1/2}T_{\rm w}, \quad {\rm w} \in {\rm \bf W}.
\] 
The sub-algebra $\mathcal{H}(K,\lambda_{\eta})$ inherits a left action on $\mathcal{H}(K,\check{\lambda}_{\eta})$
given by $C \bullet D:= D \star \check{C}$ for $C \in \mathcal{H}(K,\lambda_{\eta})$ and $D \in \mathcal{H}(K,\check{\lambda}_{\eta})$.
We summarize the properties of $\mathcal{H}(K,\lambda_{\eta})$ from \cite[\S 11.5 and \S 11.6]{BK98}.
\begin{lemma}
\label{BK-basis}
The set $\{T_{{\rm I}_2},T_{{\rm w}_0} \}$ is a $\mathbb{C}$-basis of $\mathcal{H}(K,\lambda_{\eta})$.
Specifically, $\dim_{\mathbb{C}}(\mathcal{H}(K,\lambda_{\eta}))=2$.
\end{lemma}

Since any element ${\rm w}$ in ${\rm \bf W}$ certainly intertwines the representation $\lambda_{\eta}$,  \cite[Lemma 2.3]{Kut04} (cf. \cite[(11.6)]{BK98}) implies that
$\mathcal{I}_K(\lambda_{\eta})=J_{\eta} \cup J_{\eta} {\rm w}_0 J_{\eta}$. In addition, we observe from \cite[Lemma 3.2.4]{Kim16} that
\[
 \{ g \in G \,|\, \text{There is $f \in \mathcal{F}(s,\widetilde{\eta})^{\lambda_{\eta}}$ with $f(g) \neq 0$} \}=\{g\in G \,| \, {\rm Hom}_{B \cap J^g}(\widetilde{\eta}\otimes {\textbf 1}_U,\widetilde{\eta}^g) \neq 0 \} 
 \]
is a subset of $B\mathcal{I}_K(\lambda_{\eta})$, because $G=BK$. Hence, any $f \in \mathcal{F}(s,\widetilde{\eta})^{\lambda_{\eta}}$  is determined by its restriction to $\mathcal{I}_K(\lambda_{\eta})$. 
In this regard, we define 
\begin{equation}
\label{compactisom}
\iota : {\rm Hom}_{J_{\eta}}(W,\mathcal{F}(s,\widetilde{\eta})) \rightarrow \mathcal{H}(K,\check{\lambda}_{\eta})
\end{equation}
by
$  \iota (\phi)(k)(w)=\phi(w)(k)$ for  $\phi \in {\rm Hom}_{J_{\eta}}(W,\mathcal{F}(s,\widetilde{\eta})), k \in J_{\eta}KJ_{\eta}$, and $w \in W$.

\begin{lemma}$(${\rm Cf.} \cite[Lemma 3.2.6]{Kim16}$)$
The map $\iota$ is well defined and is a homomorphism of left $\mathcal{H}(K,\lambda_{\eta})-$modules.
\end{lemma}
\begin{proof}
Let $j,j' \in J_{\eta}$ and ${\rm w} \in {\rm \bf W}$. Thanks to \ref{Iwahori}, we decompose $j=j_Uj_Tj_{\overline{U}}$ with $j_U \in J_{\eta} \cap U$, $j_T \in J_{\eta} \cap T$, and $j_{\overline{U}} \in J_{\eta} \cap \overline{U}$. We see that elements ${\rm w}^{-1}j_{\overline{U}}{\rm w}j'$ intertwine the character $\lambda_{\eta}$. With these in hand, we have
\begin{multline*}
 \iota(\phi)(j{\rm w}j')(w)=\lambda_{\eta}(j_T)  \phi(w)(j_{\overline{U}}{\rm w}j')=\lambda_{\eta}(j_T)(\rho({\rm w}^{-1}j_{\overline{U}}{\rm w}j')\phi)(w)({\rm w})\\
 =\lambda_{\eta}(j_T)\phi(\lambda_{\eta}({\rm w}^{-1}j_{\overline{U}}{\rm w}j')w)({\rm w}).
\end{multline*}
Appealing to \ref{kernel}, $\lambda_{\eta}({\rm w}^{-1}j_{\overline{U}}{\rm w})$ is trivial, and it becomes
\[
 \iota(\phi)(w)(j{\rm w}j')=\lambda_{\eta}(j_T)\phi(\lambda_{\eta}(j')w)({\rm w})=\lambda_{\eta}(j_T) \iota(\phi)({\rm w})\lambda_{\eta}(j')(w)
 =\lambda_{\eta}(j) \iota(\phi)({\rm w})\lambda_{\eta}(j')(w)
\]
having used the fact that $\lambda_{\eta}(j_T)=\lambda_{\eta}(j)$. We confirm that $\iota$ is well defined, that is to say, $\iota (\phi)$ belongs to $\mathcal{H}(K,\check{\lambda}_{\eta})$.

\par
It remains to show that $\iota$ is a homomorphism of left $\mathcal{H}(K,\lambda_{\eta})-$modules. Upon making the change of variables $g \mapsto k^{-1}g$, for $f \in \mathcal{H}(K,\lambda_{\eta})$,
$\iota(  \pi(f)\phi)(k)(w)$ equals to
\begin{multline*}
 ( \pi(f)\phi)(w)(k)=\int_G \rho(g)\phi(f(g)^{\vee}w)(k) \,dg=\int_G \phi(f(g)^{\vee}w)(kg) \,dg
 =\int_G \phi(\check{f}(g^{-1})w)(kg) \,dg
 \\
 =\int_G \iota(\phi)(kg)(\check{f}(g^{-1})w)\,dg
 =\int_G \iota(\phi)(g)(\check{f}(g^{-1}k)w)\,dg
 =( \iota( \phi) \star \check{f}  ) (k)(w)=(f \bullet \iota( \phi) )(k)(w)
 \end{multline*}
from which the desired conclusion follows. It is noteworthy that the support of all these integrals above is actually in $K$.
\end{proof}

The following proposition can be thought of as
the ${\rm SL}_2(F)$-analogue of \cite[Proposition 5.9]{JK21} and \cite[Lemma 3.2.9]{Kim16}. 

\begin{proposition}
\label{two-dim}
As left $\mathcal{H}(K,\lambda_{\eta})-$modules, $\mathcal{H}(K,\lambda_{\eta})$  is isomorphic to $\mathcal{F}(s,\widetilde{\eta})_{\lambda_{\eta}}$.
Consequently, the $\lambda_{\eta}$-isotypic subspace $\mathcal{F}(s,\widetilde{\eta})^{\lambda_{\eta}}$ is two dimensional with a $\mathbb{C}$-basis $\{  f_{{\rm I}_2,\widetilde{\eta}}, f_{{\rm w}_0,\widetilde{\eta}} \}$. 
\end{proposition}

\begin{proof}
To verify that \eqref{compactisom} is an isomorphism, it is sufficient to check that it takes a basis of $\mathcal{H}(K,\lambda_{\eta})$
to a basis of $\mathcal{F}(s,\widetilde{\eta})_{\lambda_{\eta}}$. 
The algebra $\mathcal{H}(K,\check{\lambda}_{\eta})$ is isomorphic to the left regular representation of $\mathcal{H}(K,\lambda_{\eta})$ via an isomorphism $\iota' :\mathcal{H}(K,\check{\lambda}_{\eta}) 
\rightarrow \mathcal{H}(K,\lambda_{\eta})$ given by $\iota'(D)=\check{D}$ for $D \in \mathcal{H}(K,\check{\lambda}_{\eta})$.
This is then immediate from Lemma \ref{BK-basis}, since 
$\iota( f_{{\rm w},\widetilde{\eta}})=\check{T}_{\rm w} $ for $\rm w \in {\rm \bf W}$.
\end{proof}

Since $U,\overline{U} \cong F$, we may identify the measure on $U$ and $\overline{U}$ with the additive measure $dx$ on $F$,
and we take $d^{\times} x$ to be $dx / |x|$ on $F^{\times}$. The Haar measure $dx$ is normalized so that $J_{\eta} \cap U \cong \mathfrak{o}_F$
has volume one. We define a standard $G$-intertwining operator $A(s,\chi) : \mathcal{F}(s,\chi) \rightarrow \mathcal{F}(-s,\chi^{-1})
$ by
\begin{equation}
\label{intertwining-integral}
   A(s,\chi,{\rm w}_0)(f)(g)=\int_{U} f({\rm w}_0 ug) \, du,
\end{equation}
for all $f \in  \mathcal{F}(s,\chi)$. The integral converges absolutely for ${\rm Re}(s) \gg 0$ and defines 
a rational function on a non-empty Zariski open subset of the complex torus $X(T)$. 
The character $\chi$ is said to be {\it regular} if $\chi \neq \chi^{-1}$. If $\iota_B^G(\chi\otimes\nu_s)$ is irreducible almost everywhere in $s$,
every $G$-morphism from $\mathcal{F}(s,\chi)$ to $\mathcal{F}(-s,\chi^{-1})$
is a scalar multiple of $ A(s,\chi,{\rm w}_0)$. We set $\epsilon=\widetilde{\eta}(-1)$. 
Our calculation is inspired from the idea in \cite{JK21}.

\begin{proposition}
\label{intertwining}
Let $\widetilde{\eta}$ be a non-trivial ramified quadratic character of $F^{\times}$ and assume that $\nu_s$ is regular.
Then we have
\[
 A(s,\widetilde{\eta},{\rm w}_0)(f_{{\rm I}_2,\widetilde{\eta}})=\epsilon \cdot {\rm vol}(\overline{U} \cap J_{\eta})f_{{\rm w}_0,\widetilde{\eta}^{-1}}=\epsilon q^{-n_{\eta}}f_{{\rm w}_0,\widetilde{\eta}^{-1}}
\]
and
\[
 A(s,\widetilde{\eta},{\rm w}^{-1}_0)(f_{{\rm w}_0,\widetilde{\eta}})=\epsilon \cdot {\rm vol}(U \cap J_{\eta})f_{{\rm I}_2,\widetilde{\eta}^{-1}}
= \epsilon f_{{\rm I}_2,\widetilde{\eta}^{-1}}.
\]
\end{proposition}

\begin{proof}
Since $\mathfrak{R}^{\mathfrak{s}_{\eta}}(G)$ is equivalent to $\mathcal{H}(G,\lambda_{\eta})-$Mod, we have
\[
  {\rm Hom}_{G}(\mathcal{F}(s,\widetilde{\eta}), \mathcal{F}(-s,\widetilde{\eta}^{-1}))
  \cong {\rm Hom}_{\mathcal{H}(G,\lambda_{\eta})}(\mathcal{F}(s,\widetilde{\eta})_{\lambda_{\eta}}, \mathcal{F}(-s,\widetilde{\eta}^{-1})_{\lambda_{\eta}}).
\]
Therefore, $A(s,\widetilde{\eta},{\rm w}_0)$ induces an intertwining map $\mathcal{F}(s,\widetilde{\eta})_{\lambda_{\eta}} \rightarrow \mathcal{F}(-s,\widetilde{\eta}^{-1})_{\lambda_{\eta}}$, which we by abuse of notation we will again denote by $A(s,\widetilde{\eta},{\rm w}_0)$. We prove equalities when the integral in \eqref{intertwining-integral} converges absolutely, that is to say, the real part of $s$ is sufficiently large,
and then extend meromorphically to the entire complex plane.

\par
Now it follows from Proposition \ref{two-dim} that $A(s,\widetilde{\eta},{\rm w}_0)(f_{ {\rm I}_2,\widetilde{\eta}})=a_{{\rm I}_2}f_{{\rm I}_2,\widetilde{\eta}^{-1}}+a_{{\rm w_0}}f_{{\rm w_0},\widetilde{\eta}^{-1}}$.
We may evaluate both sides of this equation at ${\rm w}_0$ to establish
\[
 a_{{\rm w}_0}=\int_{U} f_{{\rm I}_2,\widetilde{\eta}}({\rm w}_0 u{\rm w}_0) \, du
 =\widetilde{\eta}(-1)\int_{\overline{U}} f_{{\rm I}_2,\widetilde{\eta}}(\overline{u}) \, d\overline{u}
=\epsilon \cdot {\rm vol}(\overline{U} \cap J_{\eta}).
\]
For $x \neq 0$, we have the Iwasawa decomposition:
\begin{equation}
\label{Iwasawa}
  {\rm w}_0u(x)=\begin{pmatrix} x^{-1} & 0 \\ 0 & x \end{pmatrix} \begin{pmatrix} 1 & -x \\ 0 & 1 \end{pmatrix} \begin{pmatrix} 1 & 0 \\ x^{-1}& 1 \end{pmatrix}.
\end{equation}
In particular, for $x^{-1} \in \mathfrak{p}_F^{n_{\eta}}$, we have $f_{{\rm I}_2,\widetilde{\eta}}({\rm w}_0u(x))=\widetilde{\eta}\nu_s\delta^{1/2}(x^{-1})$.
We may rewrite $x \in F^{\times}$ as $x=\varpi_F^ky$ with $k \in \mathbb{Z}$ and $y \in \mathfrak{o}_F^{\times}$, and then use the relation of additive Haar measure $d(cx)=|c|dx$ for $c \in F^{\times}$ to determine $ a_{{\rm I}_2}$:
\begin{multline}
\label{a1-evaluation}
 a_{{\rm I}_2}=\int_{U} f_{{\rm I}_2,\widetilde{\eta}}({\rm w}_0 u) \, du
 =\int_{\{x\,|\, x\neq 0,\; x^{-1} \in \mathfrak{p}_F^{n_{\eta}}\}} \widetilde{\eta}(x^{-1})|x|^{-s-1} \,dx\\
 =\int_{|x| \geq q^{n_{\eta}}} \widetilde{\eta}(x^{-1})|x|^{-s-1} \,dx
 =\sum_{n=-\infty}^{-n_{\eta}} \eta(\varpi_F^{-n})|\varpi_F^{-n}|^{-s}\int_{\mathfrak{o}_F^{\times}}\widetilde{\eta}(x) \,d^{\times}x.
\end{multline}
 But $\widetilde{\eta}$ being ramified implies that the series of equality of above integrals is zero. In other words, $ a_{{\rm I}_2}=0$.

 \par
 Let us turn our attention to $A(s,\widetilde{\eta},{\rm w}^{-1}_0)(f_{{\rm w}_0,\widetilde{\eta}})$. Owing to \Cref{two-dim}, we know that 
 \[
 A(s,\widetilde{\eta},{\rm w}^{-1}_0)(f_{{\rm w}_0,\widetilde{\eta}})=b_{{\rm I}_2}f_{{\rm I}_2,\widetilde{\eta}^{-1}}+b_{{\rm w_0}}f_{{\rm w_0},\widetilde{\eta}^{-1}}.
 \]
 Our aim is to determine unknown coefficients $b_{{\rm I}_2}$ and $b_{{\rm w}_0}$ precisely. To this end, we evaluate both sides of the above equation at ${\rm I}_2$.
 The test function $f_{{\rm w}_0,\widetilde{\eta}}$ is supported on $B{\rm w}_0J_{\eta}$ so that
 \[
  b_{{\rm I}_2}=\int_{U}f_{{\rm w}_0,\widetilde{\eta}}   ({\rm w}^{-1}_0 u) \, du=\widetilde{\eta}(-1)\int_{U}f_{{\rm w}_0,\widetilde{\eta}}   ({\rm w}_0 u) 
  =\epsilon\cdot {\rm vol}(U \cap J_{\eta}),
 \]
 thereby providing $b_{{\rm I}_2}=\widetilde{\eta}(-1)$. Just as in \eqref{a1-evaluation}, we use the Iwasawa decomposition \eqref{Iwasawa} (cf. \eqref{Bruhat}) to arrive at
\begin{multline*}
  b_{{\rm w}_0}=\int_{U} f_{{\rm I}_2,\widetilde{\eta}}({\rm w}_0^{-1} u {\rm w}_0) \, du
 =\int_{\{x\,|\,x\neq 0,\; x^{-1} \in \mathfrak{o}_F \}} \widetilde{\eta}(x^{-1})|x|^{-s-1} \,dx \\
 =\sum_{n \leq 0} \eta(\varpi_F^{-n})|\varpi_F^{-n}|^{-s}\int_{\mathfrak{o}_F^{\times}}\widetilde{\eta}(x) \,d^{\times}x.
\end{multline*}
Once again, $\widetilde{\eta}$ being ramified forces that the innermost integral over $\mathfrak{o}_F^{\times}$ is equal to zero, which yields $b_{{\rm w}_0}=0$, as requested.
\end{proof}

The {\it Plancherel constant} is a scalar valued function $\mu(s,\widetilde{\eta}) \in \mathbb{C}$ attached to $\widetilde{\eta}$ is by the defining relation \cite[\S 6]{Sha84}
\[
 A(-s,\widetilde{\eta}^{-1},{\rm w}^{-1}_0) \circ A(s,\widetilde{\eta},{\rm w}_0)=\mu(s,\widetilde{\eta})^{-1} \cdot id_{\mathcal{F}(s,\widetilde{\eta})}
\]
on a Zariski open dense subset of $\mathbb{C}$. It is a rational function in $q^{-s}$ and clearly depends on the measure defining intertwining operators. 

\begin{theorem}[The Plancherel Constant I]
\label{Plancherel}
 Let $\widetilde{\eta}$ be a non-trivial ramified quadratic character of $F^{\times}$ of level $n_{\eta}$ with $\eta=\widetilde{\eta} \restriction_{\mathfrak{o}^{\times}_F}$. Then we have
\[
  \mu(s,\widetilde{\eta})={\rm vol}(\overline{U} \cap J_{\eta})^{-1}{\rm vol}(U \cap J_{\eta})^{-1}=q^{n_{\eta}}.
\]
\end{theorem}

\begin{proof}
We apply Proposition \ref{intertwining} twice.
\end{proof}

The Plancherel constant in Theorem \ref{Plancherel} coincides with that in split cases \cite[\S 3.3]{KK} on the common Zariski open dense subset of $\mathbb{C}$ (Refer to \cite[Theorem 4.5]{KM09} for the precise description), and 
this recovers the work by Kutzko and Morris \cite[Theorem 4.5.(2)-(ii)]{KM09}.

\section{Local Coefficients and Gelfand-Graev Representations}
\label{sec3-Whittaker}

\subsection{The Local Coefficient}
\label{SL2-LC}

Let $\psi$ be a non-trivial additive character of $F$ trivial on $\mathfrak{o}_F$ but not on $\mathfrak{p}_F^{-1}$, and $\chi$ a character of $F^{\times}$.
It is a theorem of Rodier \cite{Rod72} that the dimension of the space of $\psi$-Whittaker functionals on $\mathcal{F}(s,\chi)$ 
is one. We may define a basis vector $\Omega_s$ for the $\psi$-Whittaker functionals on $\mathcal{F}(s,\chi)$  by the formula
\[
  \Omega_s(f)=\int_{U} f({\rm w}_0 u) \psi^{-1}(u) \,du.
\]
This integral may not converge for all $f$ but can be extended to the whole space as a principal value integral.
We also have the following convenient reinterpretation for $\Omega_s$ as a principal value integral; Given a compact open subgroup
$K_{\ast}$ of $G$, there exists a suitably large compact open subgroup $U_{\ast} \subset U$ such that
\[
 \Omega_s(f)=\int_{U_{\ast}} f({\rm w}_0 u) \psi^{-1}(u) \,du
\]
for all $s$ and for all $f \in \mathcal{F}(s,\chi)^{K_{\ast}}$. We similarly define $\Omega'_s$ on  $\mathcal{F}(-s,\chi^{-1})$
via
\[
  \Omega'_s(f)=\int_{U} f({\rm w}_0 u) \psi^{-1}(u) \,du
\]
as a principal value integral in the above sense. Appealing to the aforementioned result of Rodier \cite{Rod72}, there exists a non-zero constant $C_{\psi}(s,\chi)$
called the {\it Langlnads-Shahidi local coefficient} satisfying
\begin{equation}
\label{LS-Coefficients}
 C_{\psi}(s,\chi)( \Omega'_s \circ A(s,\chi,{\rm w}_0))= \Omega_s.
\end{equation}
For $c \in F$, the {\it Gauss sum} attached to $\chi$ is defined by
\[
  \tau(\chi,\psi,c)=\int_{\mathfrak{o}_F^{\times}} \chi(x^{-1}) \overline{\psi(cx)} \, dx.
\]

\begin{theorem}[The Local Coefficient]
\label{LocalCoefficient}
 Let $\widetilde{\eta}$ be a non-trivial ramified quadratic character of $F^{\times}$ of level $n_{\eta}$ with $\eta=\widetilde{\eta} \restriction_{\mathfrak{o}^{\times}_F}$. 
Then we have
\[
  C_{\psi}(s,\widetilde{\eta})=\widetilde{\eta}(-\varpi_F^{n_{\eta}})\tau(\eta,\psi,\varpi_F^{-n_{\eta}})q^{-n_{\eta}(s-1)}.
\]
\end{theorem}

\begin{proof} Since $C_{\psi}(s,\widetilde{\eta})$ is a rational function in $\mathbb{C}(q^{-s})$, it suffices to prove the assertion on a Zariski open dense subset of $X(T)$.
In particular, we impose the assumption that $\mathcal{F}(s,\widetilde{\eta})$ and $\mathcal{F}(-s,\widetilde{\eta}^{-1})$ are all irreducible.
Our argument is rather close in spirit to the work by Krishnamurthy and Kutzko \cite[Proposition 3.1]{KK}. Evaluating one side of \eqref{LS-Coefficients} at $f_{{\rm I}_2,\widetilde{\eta}}$ is straightforward, since the function $f_{{\rm w}_0,\widetilde{\eta}}$ is supported on $B{\rm w}_0J_{\eta}$.
Indeed we see that
\[
  {\rm w}_0 U \cup B{\rm w}_0 J_{\eta}= {\rm w}_0 U \cup {\rm w}_0 \overline{B} J_{\eta}={\rm w}_0(U \cap \overline{B}J_{\eta})={\rm w}_0(U \cap J_{\eta})
\]  
from which it follows that
\begin{multline*}
 ( \Omega'_s \circ A(s,\widetilde{\eta},{\rm w}_0))(f_{{\rm I}_2,\widetilde{\eta}})=\widetilde{\eta}(-1){\rm vol}(\overline{U} \cap J_{\eta})\Omega'_s(f_{{\rm w}_0,\widetilde{\eta}^{-1}}) \\
 =\widetilde{\eta}(-1){\rm vol}(\overline{U} \cap J_{\eta})\int_{J_{\eta} \cap U} f_{{\rm w}_0,\widetilde{\eta}^{-1}}({\rm w}_0 u) \psi^{-1}(u) \,du
  =\widetilde{\eta}(-1){\rm vol}(\overline{U} \cap J_{\eta})\int_{J_{\eta} \cap U}  \psi^{-1}(u) \,du.
\end{multline*}
We know that $\psi$ is trivial on $J_{\eta} \cap U$, and consequently, 
\begin{equation}
\label{Side}
 ( \Omega'_s \circ A(s,\widetilde{\eta},{\rm w}_0))(f_{{\rm I}_2,\widetilde{\eta}})=\widetilde{\eta}(-1){\rm vol}(\overline{U} \cap J_{\eta}){\rm vol}(U \cap J_{\eta})=\widetilde{\eta}(-1)q^{-n_{\eta}}.
 \end{equation}
 This brings us to the central issue of computing the other side $\Omega_s(f_{{\rm I}_2,\widetilde{\eta}})$. In contrast to  $f_{{\rm w}_0,\widetilde{\eta}}$, 
 this is not immediate, because $f_{{\rm I}_2,\widetilde{\eta}}$ is supported near the identity element.
 To this end, we deduce from \eqref{Iwasawa} that $f_{{\rm I}_2,\widetilde{\eta}}({\rm w}_0u(x))$ is $\widetilde{\eta}\nu_s\delta^{1/2}(x^{-1})$, if $x^{-1} \in \mathfrak{p}_F^{n_{\eta}}$, and $0$, otherwise. Then $\Omega_s(f_{{\rm I}_2,\widetilde{\eta}})$ equals to
 \[
  \int_{\{ \mathfrak{p}_F^{-m}-\{0\} \} \cap \{ x\,|\, x^{-1} \in \mathfrak{p}_F^{n_{\eta}}\}} \widetilde{\eta}(x^{-1}) |x|^{-s-1} \psi^{-1}(x) \,dx
 \]
 for some large positive integer $m \gg 0$. For convenience, let $\mathcal{D}$ denote 
 \[
 \{ \mathfrak{p}_F^{-m}-\{0\} \} \cap \{ x\,|\, x^{-1} \in \mathfrak{p}_F^{n_{\eta}} \}.
 \]
 For any integer $r$, let $\mathcal{D}_r$ denote the {\it shell},
$   \{ x \in \varpi_F^r\mathfrak{o}_F^{\times} \,|\,  x^{-1} \in \mathfrak{p}_F^{n_{\eta}}  \}$.
Then our domain of the integration can be decomposed as  shells $\mathcal{D}=\cup_{-m \leq r \leq -n_{\eta}} \mathcal{D}_r$.
The crux of the proof of \cite[Proposition 3.1]{KK} is that only the last shell $\mathcal{D}_{-n_{\eta}}=\varpi_F^{-n_{\eta}}\mathfrak{o}_F^{\times}$ contributes to $\Omega_s(f_{{\rm I}_2,\widetilde{\eta}})$. Assembling all of this information, we achieve
\begin{multline}
\label{DualSide}
 \Omega_s(f_{{\rm I}_2,\widetilde{\eta}})=\int_{\varpi_F^{-n_{\eta}}\mathfrak{o}_F^{\times}}  \widetilde{\eta}(x^{-1}) |x|^{-s} \psi^{-1}(x) \,d^{\times}x
 =\widetilde{\eta}(\varpi_F^{-n_{\eta}})q^{-n_{\eta}s} \int_{\mathfrak{o}_F^{\times}} \eta(x^{-1}) \psi^{-1}(\varpi_F^{-n_{\eta}} x) d^{\times} x \\
 =\widetilde{\eta}(\varpi_F^{-n_{\eta}})q^{-n_{\eta}s}\tau(\eta,\psi,\varpi_F^{-n_{\eta}}).
\end{multline}
We can draw the conclusion from \eqref{Side} combined with \eqref{DualSide}.
\end{proof}

This result should be compared with the work of Shahidi \cite[Lemma 4.4]{Sha78}. In addition, the local coefficient coincides with the corresponding Hecke--Tate local $\gamma$-factor \cite{Tat77},
which can be viewed as Bump and Friedberg exterior square local factors for ${\rm GL}_1(F) \cong F^{\times}$ (cf. \cite[Proof of Theorem 5.4]{Mat15}). Indeed, Theorem \ref{LocalCoefficient} confirms Conjecture 4 in \cite{BF89}.

\par
Let $P(X) \in \mathbb{C}[X]$ be the unique polynomial satisfying $P(0)=1$ such that $P(q^{-s})$ is the numerator of $C_{\psi}(s,\chi)$.
Whenever $\chi$ is unitary, the local $L$-factor is defined by
\[
  L(s,\chi):=\frac{1}{P(q^{-s})}.
\]
The local  $\varepsilon$-factor is defined to satisfy the relation:
\[
  C_{\psi}(s,\chi)=\varepsilon(s,\chi,\psi)\frac{L(1-s,\chi^{-1})}{L(s,\chi)}.
\]

\begin{corollary}[Local Factors] Let $\widetilde{\eta}$ be a non-trivial ramified quadratic character of $F^{\times}$. Then we have
\[
  \varepsilon(s,\widetilde{\eta},\psi)= C_{\psi}(s,\widetilde{\eta}) \quad \text{and} \quad  L(s,\widetilde{\eta})=1.
\]
\end{corollary}

Corollaries \ref{FE} and \ref{SquareRoot} match with the corresponding formul\ae{} \cite[Proposition 3.1.1]{Sha81} and \cite[\S Introduction]{Sha84}. 

\begin{corollary}[The Functional Equation] 
\label{FE}
 Let $\widetilde{\eta}$ be a non-trivial ramified quadratic character of $F^{\times}$. Then we have
\[
 C_{\psi}(s,\widetilde{\eta})C_{\psi}(1-s,\widetilde{\eta}^{-1})=\epsilon.
 \]
\end{corollary}

\begin{proof}
The functional equation for the local constant $ \varepsilon(s,\widetilde{\eta},\psi)$ (cf. \cite[Corollary 23.4.2]{BH06}) produces  
\[
C_{\psi}(s,\widetilde{\eta})C_{\psi}(1-s,\widetilde{\eta}^{-1})
=\varepsilon(1/2,\widetilde{\eta},\psi)\varepsilon(1/2,\widetilde{\eta}^{-1},\psi)=\widetilde{\eta}(-1).
\qedhere \]
\end{proof}

The local coefficient is related to the Plancherel constant which is more or less saying that ``the square root of the local coefficient equals the associated Plancherel constant" as given in Corollary \ref{SquareRoot}. 

\begin{corollary} [The Plancherel Constant II]
\label{SquareRoot}
 Let $\widetilde{\eta}$ be a non-trivial ramified quadratic character of $F^{\times}$.  Then we have
\[
  \mu(s,\widetilde{\eta})=C_{\psi}(s,\widetilde{\eta}) C_{\psi^{-1}}(-s,\widetilde{\eta}^{-1})=|C_{\psi}(s,\widetilde{\eta}) |^2. 
\]
\end{corollary}

\begin{proof} We observe from \cite[(23.6.3)]{BH06} that
\[
\tau(\eta,\psi,\varpi_F^{-n_{\eta}})\tau(\eta^{-1},\psi^{-1},\varpi_F^{-n_{\eta}})=|\tau(\eta,\psi,\varpi_F^{-n_{\eta}})|^2=q^{-n_{\eta}}.
\qedhere \]
\end{proof}

\subsection{The Gelfand-Graev Representation}
\label{GGrep}
We take this occasion to explore the structure of the Gelfand-Graev space, which extends the results of \cite{CS18} and \cite{MP21}.
The Gelfand-Graev representation ${\text {\rm c-ind}}_U^G{\psi}$ \cite[\S 4.1]{CS18}  is provided by the space of smooth functions $f : G \rightarrow \mathbb{C}$
which are compactly supported modulo $U$. They also satisfy
\begin{equation}
\label{Whittaker-invariance}
 f(ug)=\psi(u)f(g) \quad \text{for all $u \in U$ and $g \in G$.}
\end{equation}
Let ${\text {\rm Ind}}_U^G{\psi}$ denote the full space of smooth functions $f : G \rightarrow \mathbb{C}$ which satisfy \eqref{Whittaker-invariance}.
As a potential application (cf. \cite[\S 5.3.2 (5.25)]{JK21}), the contragredient $({\text {\rm c-ind}}_U^G\,{\overline{\psi}})^{\vee} \cong {\text {\rm Ind}}_U^G{\psi}$ of ${\text {\rm c-ind}}_U^G\,{\overline{\psi}}$
appears in the target space of the {\it Whittaker map} $\omega_s : \mathcal{F}(s,\widetilde{\eta}) \rightarrow {\text {\rm Ind}}_U^G{\psi}$
corresponding to $\Omega_s$ via Frobenius reciprocity. 
For $w \in W$, we define the functions $\varphi^{\psi}_{{\rm I}_2,\eta}$ and $\varphi^{\psi}_{{\rm w}_0,\eta}:G \rightarrow \mathbb{C}$ in the $\lambda_{\eta}$-co-invariant Gelfand-Graev space $({\text {\rm c-ind}}_U^G{\psi})^{\lambda_{\eta}}$ given by
\[
  \varphi^{\psi}_{{\rm I}_2,\eta}(w)(g)=
  \begin{cases}
 \psi(u)\lambda_{\eta}(j),  \ & \text{if} \quad g=uj \in U J_{\eta} \\
  0, & \text{otherwise,}
  \end{cases}
\]
and
\[
  \varphi^{\psi}_{{\rm w}_0,\eta}(w)(g)=
  \begin{cases}
\psi(u)\lambda_{\eta}(j), & \text{if} \quad g=u{\rm w}_0j \in U {\rm w}_0 J_{\eta}\\
  0, & \text{otherwise.}
  \end{cases}
\]
Given a smooth (complex) representation $(\pi,V)$ of $G$, let $V_{\overline{U}}$ be the maximal quotient of $V$ on which $\overline{U}$ acts trivially. 
According to \citelist{\cite{Bor76}*{Lemma 4.7} \cite{CS18}*{Proposition 4.2}}  (cf. \cite[Lemma 10.3]{BK98}) accompanied by \cite[Lemma 2.3]{BH03}, the usual projection from ${\text {\rm c-ind}}_U^G{\psi}$ onto 
${\text {\rm c-ind}}_1^T{\mathbb{C}} \cong C_c^{\infty}(T) $ defined by
\[
 \mathcal{P}_{\delta}: f \mapsto f_{\overline{U}}(t)=\delta^{1/2}(t) \int_{\overline{U}} f(t\overline{u})\, d\overline{u}, \quad t \in T
\]
descends to an isomorphism $\mathcal{S}_{\delta} : ({\text {\rm c-ind}}_{U}^G{\psi})_{\overline{U}} \rightarrow C_c^{\infty}(T)$ as $C_c^{\infty}(T)$-modules. 
Let ${\rm ch}^{\eta}_{T(\mathfrak{o}_F)} \in (C_c^{\infty}(T))^{\eta}$ be a test function supported on $T(\mathfrak{o}_F)$ such that ${\rm ch}^{\eta}_{T(\mathfrak{o}_F) }(t)=\mathrm{vol}(\overline{U} \cap J_{\eta})  \eta(t)$ for all $t \in T(\mathfrak{o}_F)$.
Owing to \cite[Theorem 1]{MP21} (cf. \cite[The discussion preceding Lemma 4.3]{CS18}), 
$\mathcal{P}_{\delta}$ in turn endows the isomorphism as ${\mathcal H}(T,\eta)$-modules of $({\text {\rm c-ind}}_U^G{\psi})^{\lambda_{\eta}}$
with $(C_c^{\infty}(T))^{\eta} \cong {\mathcal H}(T,\eta)$, where ${\rm ch}^{\eta}_{T(\mathfrak{o}_F)}$ is a generator. 

\begin{proposition}
\label{Whittaker-Hecke}
Let $\widetilde{\eta}$ be a non-trivial ramified quadratic character of $F^{\times}$ with $\eta=\widetilde{\eta} \restriction_{\mathfrak{o}^{\times}_F}$. Then we have
 \begin{enumerate}[label=$(\mathrm{\arabic*})$]
 \item\label{abel} $\mathcal{S}_{\delta} ( \varphi^{\psi}_{{\rm I}_2,\eta}-\epsilon^{3/2} \cdot {\rm vol}(J_{\eta}{\rm w}_0J_{\eta})^{-1/2}\varphi^{\psi}_{{\rm w}_0,\eta} )={\rm ch}^{\eta}_{T(\mathfrak{o}_F)}$.
 \item\label{cpt} $ T^{\ast}_{\rm w} \star (  \varphi^{\psi}_{{\rm I}_2,\eta}-\epsilon^{3/2}\cdot {\rm vol}(J_{\eta}{\rm w}_0J_{\eta})^{-1/2}\varphi^{\psi}_{{\rm w}_0,\eta})=(-1)^{\ell(({\rm w})}  (\varphi^{\psi}_{{\rm I}_2,\eta}-\epsilon^{3/2}\cdot {\rm vol}(J_{\eta}{\rm w}_0J_{\eta})^{-1/2}\varphi^{\psi}_{{\rm w}_0,\eta}) $ for ${\rm w} \in {\rm \bf W}$.
  \end{enumerate}
\end{proposition}

\begin{proof}
The first assertion \ref{abel} is immediate from Claim located in the halfway of the proof of \cite[Lemma 4.1]{CS18}.
\par
As for \ref{cpt}, we take $w \in W$ and the function $T^{\ast}_{\rm w} \star ( \varphi^{\psi}_{{\rm I}_2,\eta}-\epsilon^{3/2} \cdot {\rm vol}(J_{\eta}{\rm w}_0J_{\eta})^{-1/2}\varphi^{\psi}_{{\rm w}_0,\eta})(w)(g)$ is equal to
\begin{multline*}
\int_G (\varphi^{\psi}_{{\rm I}_2,\eta}-\epsilon^{3/2} \cdot {\rm vol}(J_{\eta}{\rm w}_0J_{\eta})^{-1/2}\varphi^{\psi}_{{\rm w}_0,\eta})(T^{\ast}_{\rm w} (x)^{\vee}w)(gx) \,dx\\
=\int_G (\varphi^{\psi}_{{\rm I}_2,\eta}-\epsilon^{3/2} \cdot {\rm vol}(J_{\eta}{\rm w}_0J_{\eta})^{-1/2}\varphi^{\psi}_{{\rm w}_0,\eta})(\check{T}^{\ast}_{\rm w} (x^{-1})w)(gx) \,dx.
\end{multline*}
The support of the function $\varphi^{\psi}_{{\rm I}_2,\eta}-\epsilon^{3/2} \cdot {\rm vol}(J_{\eta}{\rm w}_0J_{\eta})^{-1/2}\varphi^{\psi}_{{\rm w}_0,\eta}$
is contained in $UJ_{\eta} \cup U {\rm w}_0 J_{\eta} \ni gx$. 
Upon using the support of $\check{T}^{\ast}_{\rm w}$, we find that the convolution
$T^{\ast}_{\rm w} \star ( \varphi^{\psi}_{{\rm I}_2,\eta}-\epsilon^{3/2}  \cdot {\rm vol}(J_{\eta}{\rm w}_0J_{\eta})^{-1/2}\varphi^{\psi}_{{\rm w}_0,\eta})$
is supported on
\[
 UJ_{\eta} \cup U {\rm w}_0 J_{\eta} \supseteq U((J_{\eta} {\rm w} J_{\eta}  \cup  {\rm w}_0 J_{\eta}{\rm w} J_{\eta}) \cap  \mathcal{I}_K(\lambda_{\eta})) \supseteq  (UJ_{\eta} \cup U {\rm w}_0 J_{\eta}) x^{-1}  \ni g.
\]
This permits us to simplify the computation into specializing the values at ${\rm I}_2$ and ${\rm w}_0$. 

\par 
We consider the case that ${\rm w}={\rm I}_2$. We know that $({\text {\rm c-ind}}_U^G{\psi})^{\lambda_{\eta}}$ is a ${\mathcal H}(G,\eta)$-algebra, while $T^{\ast}_{{\rm I}_2}$
is the identity in this Hecke algebra ${\mathcal H}(G,\eta)$ so must surely act as the identity on itself $({\text {\rm c-ind}}_U^G{\psi})^{\lambda_{\eta}}$. Henceforth we conclude that
\[
 T^{\ast}_{{\rm I}_2} \star ( \varphi^{\psi}_{{\rm I}_2,\eta}-\epsilon^{3/2} \cdot  {\rm vol}(J_{\eta}{\rm w}_0J_{\eta})^{-1/2}\varphi^{\psi}_{{\rm w}_0,\eta})=\varphi^{\psi}_{{\rm I}_2,\eta}-\epsilon^{3/2} \cdot {\rm vol}(J_{\eta}{\rm w}_0J_{\eta})^{-1/2}\varphi^{\psi}_{{\rm w}_0,\eta}.
  \]

\par
We bring our attention to the effect of intertwiners $T^{\ast}_{{\rm w}_0}$ on the generator. 
We treat the term $ T^{\ast}_{{\rm w}_0} \star \varphi^{\psi}_{{\rm I}_2,\eta}$ first.
 The value $(T^{\ast}_{{\rm w}_0} \star \varphi^{\psi}_{{\rm I}_2,\eta})({\rm I}_2)$ is $0$,
as supports $J_{\eta}{\rm w}_0J_{\eta}$ and $UJ_{\eta}$ of  $T^{\ast}_{{\rm w}_0}$ and  $\varphi^{\psi}_{{\rm I}_2,\eta}$, respectively, are disjoint.
Now thanks to the Iwahori factorization of $J_{\eta}$, we may identify  
\[
J_{\eta} \slash (J_{\eta} \cap {\rm w}_0J_{\eta}{\rm w}^{-1}_0)=\begin{pmatrix} 1 & \mathfrak{o}_F \slash \mathfrak{p}_F^{n_{\eta}} \\ & 1 \end{pmatrix}
\]
from which we deduce that  $( \pi(T^{\ast}_{{\rm w}_0})  \varphi^{\psi}_{{\rm I}_2,\eta})(w)({\rm w}_0)$ is equal to
\begin{multline*}
   \int_{J_{\eta} {\rm w}_0 J_{\eta}} \varphi^{\psi}_{{\rm I}_2,\eta}(\check{T}^{\ast}_{{\rm w}_0}(x^{-1})w) ({\rm w}_0x) \, dx  
   =\sum_{j \in J_{\eta} \slash (J_{\eta} \cap {\rm w}_0J_{\eta}{\rm w}_0^{-1})} \int_{J_{\eta}}   \varphi^{\psi}_{{\rm I}_2,\eta}(\check{T}^{\ast}_{{\rm w}_0}({j'}^{-1}{\rm w}^{-1}_0j^{-1})w) ({\rm w}_0j{\rm w}_0j') dj'\\
   = \widetilde{\eta}(-1) \sum_{j \in J_{\eta} \slash (J_{\eta} \cap {\rm w}_0J_{\eta}{\rm w}^{-1}_0)} \varphi^{\psi}_{{\rm I}_2,\eta}(\check{T}^{\ast}_{{\rm w}_0}({\rm w}^{-1}_0j^{-1})w) ({\rm w}_0j{\rm w}^{-1}_0).
\end{multline*}
In this circumstance, we have employed the relation ${\rm w}_0^2=-1$ and it is worthwhile to notice that the sign change occurs. Upon writing $j=u(x)$, $\varphi^{\psi}_{{\rm I}_2,\eta}(\check{T}^{\ast}_{{\rm w}_0}({\rm w}^{-1}_0j^{-1})w) ({\rm w}_0j{\rm w}^{-1}_0) \neq 0$ if and only if $x \in \mathfrak{p}_F^{n_{\eta}}$.
As a result, we end up at
\[
 ( \pi(T^{\ast}_{{\rm w}_0})  \varphi^{\psi}_{{\rm I}_2,\eta})(w)({\rm w}_0) 
 =\epsilon^{3/2} \cdot {\rm vol}(J_{\eta}{\rm w}_0J_{\eta})^{-1/2} \varphi^{\psi}_{{\rm I}_2,\eta}(w)({\rm I}_2)=\epsilon^{3/2} \cdot {\rm vol}(J_{\eta}{\rm w}_0J_{\eta})^{-1/2}.
\]

\par
Let us focus our attention to the term $T^{\ast}_{{\rm w}_0} \star  \epsilon^{3/2}  {\rm vol}(J_{\eta}{\rm w}_0J_{\eta})^{-1/2} \varphi^{\psi}_{{\rm w}_0,\eta}$.  Then the support of $ \varphi^{\psi}_{{\rm w}_0,\eta}$ tells us that
\begin{multline*}
 ( \pi(T^{\ast}_{{\rm w}_0}) \epsilon^{3/2} \cdot {\rm vol}(J_{\eta}{\rm w}_0J_{\eta})^{-1/2}  \varphi^{\psi}_{{\rm w}_0,\eta})(w)({\rm I}_2) 
= \epsilon^{3/2} \cdot {\rm vol}(J_{\eta}{\rm w}_0J_{\eta})^{-1/2}  \int_{J_{\eta} {\rm w}_0 J_{\eta}} \varphi^{\psi}_{{\rm w}_0,\eta}(\check{T}^{\ast}_{{\rm w}_0}(x^{-1})w) (x) \, dx \\
 = \epsilon^{3/2} \cdot {\rm vol}(J_{\eta}{\rm w}_0J_{\eta})^{-1/2} \sum_{j \in J_{\eta} \slash (J_{\eta} \cap {\rm w}_0J_{\eta}{\rm w}_0^{-1})} \int_{J_{\eta} }\varphi^{\psi}_{{\rm w}_0,\eta}(\check{T}^{\ast}_{{\rm w}_0}({j'}^{-1}{\rm w}^{-1}_0j^{-1})w) (j{\rm w}_0j') \,  dj'. 
 \end{multline*}
In light of $ [J_{\eta}: J_{\eta} \cap {\rm w}_0J_{\eta}{\rm w}_0^{-1}]=[J_{\eta}{\rm w}_0J_{\eta}  : J_{\eta}]$, we are led to
 \begin{multline*}
 ( \pi(T^{\ast}_{{\rm w}_0})  \epsilon^{3/2} \cdot  {\rm vol}(J_{\eta}{\rm w}_0J_{\eta})^{-1/2}  \varphi^{\psi}_{{\rm w}_0,\eta})(w)({\rm I}_2)\\
  = \epsilon^{3/2} \cdot {\rm vol}(J_{\eta}{\rm w}_0J_{\eta})^{-1/2} \sum_{j \in J_{\eta} \slash (J_{\eta} \cap {\rm w}_0J_{\eta}{\rm w}_0^{-1})} \varphi^{\psi}_{{\rm w}_0,\eta}(\check{T}^{\ast}_{{\rm w}_0}({\rm w}^{-1}_0j^{-1})w) (j{\rm w}_0) \\
=   {\rm vol}(J_{\eta}{\rm w}_0J_{\eta})^{-1} [J_{\eta}: J_{\eta} \cap {\rm w}_0J_{\eta}{\rm w}_0^{-1}] \varphi^{\psi}_{{\rm w}_0,\eta}(w)({\rm w}_0)=1.
\end{multline*}
Concerning the value  $  ( \pi(T^{\ast}_{{\rm w}_0})  \epsilon^{3/2} \cdot {\rm vol}(J_{\eta}{\rm w}_0J_{\eta})^{-1/2} \varphi^{\psi}_{{\rm w}_0,\eta})(w)(\cdot)$ at ${\rm w}_0 $, we use the fact that the integrand below is supported on $J_{\eta}{\rm w}_0J_{\eta}$
and is right invariant under $J_{\eta}$ coupled with the fact that
the map $j \mapsto j{\rm w}_0$ induces a bijection $J_{\eta} \slash (J_{\eta} \cap {\rm w}_0J_{\eta}{\rm w}_0^{-1}) \rightarrow J_{\eta}{\rm w}_0J_{\eta} \slash J_{\eta}$ (cf. \cite[Lemma 3.2]{Kut04} to see that
 \begin{multline*}
  ( \pi(T^{\ast}_{{\rm w}_0})  \epsilon^{3/2} \cdot {\rm vol}(J_{\eta}{\rm w}_0J_{\eta})^{-1/2} \varphi^{\psi}_{{\rm w}_0,\eta})(w)({\rm w}_0)
= \epsilon^{3/2} \cdot {\rm vol}(J_{\eta}{\rm w}_0J_{\eta})^{-1/2} \int_{G} \varphi^{\psi}_{{\rm w}_0,\eta}(\check{T}^{\ast}_{{\rm w}_0}(x^{-1})w) ({\rm w}_0x) \, dx\\
=  \epsilon^{3/2} \cdot{\rm vol}(J_{\eta}{\rm w}_0J_{\eta})^{-1/2}[J_{\eta}{\rm w}_0J_{\eta}: J_{\eta}] \int_{J_{\eta}} \varphi^{\psi}_{{\rm w}_0,\eta}(\check{T}^{\ast}_{{\rm w}_0}({\rm w}_0^{-1}j^{-1})w) ({\rm w}_0j{\rm w}_0) \, dj.
\end{multline*}
The integrand of the last integral is right invariant under $J_{\eta} \cap \overline{B}$. In this way, we find that
\[
 ( \pi(T^{\ast}_{{\rm w}_0})   \epsilon^{3/2} \cdot {\rm vol}(J_{\eta}{\rm w}_0J_{\eta})^{-1/2}\varphi^{\psi}_{{\rm w}_0,\eta})(w)({\rm w}_0)= \frac{{\rm vol}(J_{\eta})}{{\rm vol}(J_{\eta} \cap U)}  \int_{J_{\eta} \cap U}
  \varphi^{\psi}_{{\rm w}_0,\eta} (w) ({\rm w}_0u{\rm w}_0)\,du. 
\]
The integrand is $0$, unless $x \notin \mathfrak{o}_F^{\times}$. Then for $x \in \mathfrak{o}_F^{\times}$, we apply the Bruhat decomposition (cf. \cite[P. 606]{Kut04});
\begin{equation}
\label{Bruhat}
  \begin{pmatrix} 1 & 0 \\ x & 1 \end{pmatrix}=-\begin{pmatrix} x^{-1} & 0 \\ 0 & x \end{pmatrix} \begin{pmatrix} 1 & x \\ 0 & 1 \end{pmatrix}
  {\rm w}_0 \begin{pmatrix} 1 & x^{-1} \\ 0& 1\end{pmatrix},
\end{equation}
whence $\varphi^{\psi}_{{\rm w}_0,\eta} \left({\rm w}_0 \begin{pmatrix} 1 & x \\ 0 & 1 \end{pmatrix}{\rm w}_0 \right)=\eta(x^{-1})\psi(-x)$.
Therefore it  can be expressed in terms of Gauss sums attached to the ramified character $\widetilde{\eta}$ and then eventually this term vanishes
\[
 ( \pi(T^{\ast}_{{\rm w}_0})  \epsilon^{3/2} \cdot {\rm vol}(J_{\eta}{\rm w}_0J_{\eta})^{-1/2} \varphi^{\psi}_{{\rm w}_0,\eta})(w)({\rm w}_0)=\int_{\mathfrak{o}_F^{\times}}\eta(x^{-1})\psi(x) \varphi^{\psi}_{{\rm w}_0,\eta} (w)({\rm w}_0)dx   =0,
\]
as was to be shown. In short, we accomplish that
\[
T^{\ast}_{{\rm w}_0} \star  \varphi^{\psi}_{{\rm I}_2,\eta}=\epsilon^{3/2} \cdot {\rm vol}(J_{\eta}{\rm w}_0J_{\eta})^{-1/2} \varphi^{\psi}_{{\rm w}_0,\eta}
\quad \text{and} \quad
 T^{\ast}_{{\rm w}_0} \star \epsilon^{3/2} \cdot {\rm vol}(J_{\eta}{\rm w}_0J_{\eta})^{-1/2} \varphi^{\psi}_{{\rm w}_0,\eta}= \varphi^{\psi}_{{\rm I}_2,\eta}.
\]
and the outcome we seek for follows from this.
\end{proof}

It is noteworthy that by applying Proposition \ref{Whittaker-Hecke}-\ref{cpt} repeatedly, we get
\[
T^{\ast 2}_{{\rm w}_0} \star (  \varphi^{\psi}_{{\rm I}_2,\eta}-\epsilon^{3/2}\cdot {\rm vol}(J_{\eta}{\rm w}_0J_{\eta})^{-1/2}\varphi^{\psi}_{{\rm w}_0,\eta})=\varphi^{\psi}_{{\rm I}_2,\eta}-\epsilon^{3/2}\cdot {\rm vol}(J_{\eta}{\rm w}_0J_{\eta})^{-1/2}\varphi^{\psi}_{{\rm w}_0,\eta}
\]
which is consistent with the consequence from $T^{\ast 2}_{{\rm w}_0}=T^{\ast}_{{\rm I}_2}$  \cite[Proposition 3.3]{Kut04}.
We denote by $\mathbb{C}_{-1}$ the one-dimensional $\mathcal{H}(K,\lambda_{\eta})-$module on which $T^{\ast}_{\rm w}$ acts as $(-1)^{\ell({\rm w})}$, for ${\rm w} \in {\rm \bf W}$. 
We are now in a position to state our main theorem. 

\begin{theorem}
\label{Gelfand-Grev}
Let $\widetilde{\eta}$ be a non-trivial ramified quadratic character of $F^{\times}$ with $\eta=\widetilde{\eta} \restriction_{\mathfrak{o}^{\times}_F}$. As $\mathcal{H}(G,\lambda_{\eta})$-modules, we have isomorphisms
\[
({\text {\rm c-ind}}_U^G{\psi})^{\lambda_{\eta}} 
 \cong \mathcal{H}(G,\lambda_{\eta} ) \otimes_{{\mathcal{H}(K,\lambda_{\eta})}}  \mathbb{C}_{-1}.
\]
\end{theorem}

\begin{proof}
In virtue of Proposition \ref{Whittaker-Hecke}-\ref{cpt}, we construct an element in ${\rm Hom}_{{\mathcal{H}(K,\lambda_{\eta})}}(\mathbb{C}_{-1},({\text {\rm c-ind}}_U^G{\psi})^{\lambda_{\eta}})$ given by $1 \mapsto \varphi^{\psi}_{{\rm I}_2,\eta}-\epsilon^{3/2} \cdot {\rm vol}(J_{\eta}{\rm w}_0J_{\eta})^{-1/2}\varphi^{\psi}_{{\rm w}_0,\eta}$.
We attain the following Frobenius reciprocity
\[
  {\rm Hom}_{{\mathcal{H}(K,\lambda_{\eta})}}(\mathbb{C}_{-1},({\text {\rm c-ind}}_U^G{\psi})^{\lambda_{\eta}})
  \cong {\rm Hom}_{{\mathcal{H}(G,\lambda_{\eta})}}(\mathcal{H}(G,\lambda_{\eta} ) \otimes_{{\mathcal{H}(K,\lambda_{\eta})}} \mathbb{C}_{-1},({\text {\rm c-ind}}_U^G{\psi})^{\lambda_{\eta}}),
\]
where the element $1 \mapsto \varphi^{\psi}_{{\rm I}_2,\eta}-\epsilon^{3/2} \cdot {\rm vol}(J_{\eta}{\rm w}_0J_{\eta})^{-1/2}\varphi^{\psi}_{{\rm w}_0,\eta}$ corresponds to the element $T^{\ast}_{{\rm I}_2} \otimes 1 \mapsto \varphi^{\psi}_{{\rm I}_2,\eta}-\epsilon^{3/2} \cdot {\rm vol}(J_{\eta}{\rm w}_0J_{\eta})^{-1/2}\varphi^{\psi}_{{\rm w}_0,\eta}$. 
In summary, we form a non-zero map 
\[
\vartheta : \mathcal{H}(G,\lambda_{\eta} ) \otimes_{{\mathcal{H}(K,\lambda_{\eta})}} \mathbb{C}_{-1} \rightarrow ({\text {\rm c-ind}}_U^G{\psi})^{\lambda_{\eta}}
\]
of $ \mathcal{H}(G,\lambda_{\eta} )$-modules, which assigns  $T^{\ast}_{{\rm I}_2} \otimes 1$ to $\varphi^{\psi}_{{\rm I}_2,\eta}-\epsilon^{3/2} \cdot {\rm vol}(J_{\eta}{\rm w}_0J_{\eta})^{-1/2}\varphi^{\psi}_{{\rm w}_0,\eta}$. 

\par
It suffices to show that $\vartheta$ is an isomorphism of $\mathcal{H}(T,\eta )$-modules.
To verify it, Proposition \ref{Whittaker-Hecke}-\ref{abel} ensures that we find an isomorphism ${\mathcal H}(T,\eta) \rightarrow ({\text {\rm c-ind}}_U^G{\psi})^{\lambda_{\eta}}$ of ${\mathcal H}(T,\eta)$-modules
provided by ${\rm ch}^{\eta}_{T(\mathfrak{o}_F)}$ mapping to $\varphi^{\psi}_{{\rm I}_2,\eta}-\epsilon^{3/2} \cdot {\rm vol}(J_{\eta}{\rm w}_0J_{\eta})^{-1/2}\varphi^{\psi}_{{\rm w}_0,\eta}$. Upon invoking the isomorphism $\mathcal{H}(T,\eta)  \otimes_{\mathbb{C}}  \mathcal{H}(K,\lambda_{\eta}) \cong \mathcal{H}(G,\lambda_{\eta}) $ attributed to \cite[Proposition 3.1]{Kut04},
we end up with a series of isomorphisms as left $\mathcal{H}(T,\eta)-$modules
\[
{\mathcal{H}(T,\eta)} \cong {\mathcal{H}(T,\eta)} \otimes_{\mathbb{C}} \mathbb{C}_{-1} \cong {\mathcal{H}(T,\eta)}  \otimes_{\mathbb{C}}  {\mathcal{H}(K,\lambda_{\eta})} \otimes_{{\mathcal{H}(K,\lambda_{\eta})} } \mathbb{C}_{-1}  \cong \mathcal{H}(G,\lambda_{\eta}) \otimes_{{\mathcal{H}(K,\lambda_{\eta})} } \mathbb{C}_{-1}, 
\]
which sends  ${\rm ch}^{\eta}_{T(\mathfrak{o}_F)}$ to $T^{\ast}_{{\rm I}_2} \otimes 1$. Gathering all information, the conclusion thereby follows, because $\vartheta$ is indeed $\mathcal{H}(T,\eta )$-modules. In other words, $\mathcal{H}(G,\lambda_{\eta} ) \otimes_{{\mathcal{H}(K,\lambda_{\eta})}}  \mathbb{C}_{-1}$ and $({\text {\rm c-ind}}_U^G{\psi})^{\lambda_{\eta}}$ are free $\mathcal{H}(T,\eta )$-modules generated by $T^{\ast}_{{\rm I}_2} \otimes 1$ and $ \varphi^{\psi}_{{\rm I}_2,\eta}-\epsilon^{3/2} \cdot {\rm vol}(J_{\eta}{\rm w}_0J_{\eta})^{-1/2}\varphi^{\psi}_{{\rm w}_0,\eta}$ respectively.
\end{proof}

When the characteristic of the field is $0$ and that of the residual field is greater than $3$, the above theorem has been settled in the work by Mishra and  Pattanayak \cite[Theorem 3]{MP21}. To the best of our knowledge, Theorem \ref{Gelfand-Grev} is new when $F$ has positive characteristics, and also when it has residual characteristic $2$ or $3$.

\begin{remark}
If $(J_{\eta},\lambda_{\eta})$ is a split cover, then the injection $t_Q$ in \ref{injection} of \S \ref{TC} becomes an isomorphism \cite[Corollary 3.1]{Kut04}. Then Proposition \ref{Whittaker-Hecke}-\ref{cpt} seems to be superfluous, because the map $t_Q$ already gives rise to a series of ${\mathcal H}(T,\eta) \cong \mathcal{H}(G,\lambda_{\eta} )$-module isomorphism
\[
 ({\text {\rm c-ind}}_U^G{\psi})^{\lambda_{\eta}} \cong \mathcal{H}(G,\lambda_{\eta} ) \otimes_{{\mathcal{H}(K,\lambda_{\eta})}}  \mathbb{C}_{-1}
 \cong {\mathcal H}(T,\eta) \otimes_{{\mathcal{H}(K,\lambda_{\eta})}}  \mathbb{C}_{-1} \cong {\mathcal H}(T,\eta) \cong (C_c^{\infty}(T))^{\eta}.
\] 
\end{remark}

\begin{acknowledgements}
I am deeply indebted to Muthu Krishnamurthy for introducing his unfinished project \cite{KK} to me. Without his suggestion and countless encouragement, this paper never has come into existence.
 I also want to take this opportunity to thank to
Jack Buttcane, Brandon Hanson, and Andrew Knightly for fruitful mathematical conversation and their feedbacks on some aspects of this manuscript. 
I am grateful to the Department of Mathematics and Statistics at University of Maine for their warm hospitality, while the paper was written.
The author would like to convey our appreciation to anonymous referees for thoughtfully reading our paper and correcting many inaccuracies which significantly improved the exposition of this manuscript. 
This work was supported by the National Research Foundation of Korea (NRF) grant 
funded by the Korea government (No. RS-2023-00209992). 
\end{acknowledgements}

\begin{conflicts of interest}
The author states that there is no conflict of interest.
\end{conflicts of interest}

\begin{data availability}
This manuscript has no associated data.
\end{data availability}

 \bibliographystyle{amsplain}
 \bibliography{references}

\end{document}